\documentclass[12pt,reqno]{amsart}
\usepackage{amsmath, amsthm, amscd, amsfonts, amssymb, graphicx, color}
\usepackage[bookmarksnumbered, colorlinks, plainpages]{hyperref}

\newtheorem{theorem}{Theorem}[section]

\newtheorem{proposition}[theorem]{Proposition}

\theoremstyle{definition}

\theoremstyle{remark}

\numberwithin{equation}{section}

\DeclareMathOperator{\Cent}{Cent}

\newcommand{\Z}{\mathbb{Z}}

\begin{document}
\setcounter{page}{1}
\title[Commuting probabilities of $n$-centralizer finite rings]{Commuting probabilities of $n$-centralizer finite rings}

\author[J. Dutta, D. K. Basnet,  R. K. Nath]{Jutirekha Dutta, Dhiren Kumar Basnet and Rajat Kanti Nath
}

\address{ Department of Mathematical Sciences, Tezpur
University,  Napaam-784028, Sonitpur, Assam, India.}
\email{jutirekhadutta@yahoo.com; dbasnet@tezu.ernet.in and
\newline rajatkantinath@yahoo.com}



\keywords{ finite ring, commuting probability, $n$-centralizer rings.}  
 
\subjclass[2010]{16U70, 16U80.} 



\begin{abstract} 
Let $R$ be a finite ring. The commuting probability of  $R$, denoted by $\Pr(R)$, is the probability that any two randomly chosen elements of $R$ commute.  $R$ is called an $n$-centralizer ring if it has $n$ distinct  centralizers. In this paper, we compute $\Pr(R)$ for some $n$-centralizer finite rings.  
\end{abstract}
\maketitle

\section{Introduction}
 The commuting probability of a finite ring $R$ is the probability that a randomly chosen pair of elements of $R$ commute. We write $\Pr(R)$ to denote this probability. 
It is not difficult to see that   
\begin{equation}\label{comformula}
\Pr(R)  = \frac{|Z(R)|}{|R|} + \frac{1}{|R|^2}\underset{r \in R \setminus Z(R)}{\sum}|C_R(r)|
\end{equation}
where $C_R(r)$ and $Z(R)$ are known as centralizer of $r \in R$ and center of $R$ given by  $\{s \in R : rs = sr\}$ and $Z(R) = \underset{r \in R}{\cap}C_R(r)$ respectively. 

In 1976, MacHale \cite{dmachale} initiated the study of $\Pr(R)$. However, it gets popularity in the recent years only unlike the commuting probability of a finite group.  The study of commuting probability of a finite group was originated from the works of Erd$\ddot{\rm o}$s and Tur$\acute{\rm a}$n \cite{pEpT68} published in 1968. Many mathematicians have considered the notion of commuting probability of a finite group in their works (see \cite{DN11} and the references therein). However, only few papers are available on $\Pr(R)$ in the literature (see \cite{BM, BMS,duttabasnetnath,dmachale} for example).     
Recently, Buckley et al.  \cite{BMS} have computed $\Pr(R)$ for several families of finite rings and characterize all finite rings having $\Pr(R) \geq \frac{11}{32}$. In this paper, we compute commuting probabilities of finite $n$-centralizer rings for $n \leq 7$. Recall that a ring $R$ is called $n$-centralizer if $|\Cent(R)| = n$, where $\Cent(R) = \{C_R(x) : x \in R\}$.  The class of $n$-centralizer finite rings was introduced and studied by Dutta et al. in \cite{dbn2014, dbn2015}.

Throughout this paper $R$ denotes a finite ring.   For any subring $S$ of $R$, we write $R/S$ or $\frac{R}{S}$  to denote the additive quotient group    $(S, +)$ in   $(R, +)$. The isomorphisms considered in this paper are the additive group isomorphisms. We shall also use the fact that for any non-commutative ring $R$, the additive group $\frac{R}{Z(R)}$ is not a cyclic group (see \cite[Lemma 1]{dmachale}).

\section{Rings with known central factor}
In this section, we deduce some properties of a finite non-commutative ring $R$ having some known central factor. These results will also serve as prerequisites for  the results obtained in the next section. We begin with the following result which is obtained by MacHale \cite[Theorem 1 and 2]{dmachale} in the year $1974$. 

\begin{theorem}\label{machale1-2}
Let $R$ be a finite ring. Then $\frac{R}{Z(R)} \cong {\Z}_2 \times {\Z}_2$ if and only if $\Pr(R) = \frac{5}{8}$. Further, if $p$ is the smallest prime divisor of $|R|$ then $\frac{R}{Z(R)} \cong {\Z}_p \times {\Z}_p$ if and only if $\Pr(R) = \frac{p^2 + p - 1}{p^3}$.
\end{theorem}
If $p$ is any  prime divisor of $|R|$, not necessarily the smallest one, then we have the following result.
\begin{theorem}\label{3chdc}
Let $R$ be a ring and $\frac{R}{Z(R)} \cong {\Z}_p \times {\Z}_p$, where $p$ is a prime. Then $R$ is a $(p + 2)$-centralizer ring and $\Pr(R) = \frac{p^2 + p - 1}{p^3}$.
\end{theorem} 

\begin{proof}
We write $Z := Z(R)$. Since  $R/Z \cong {\Z}_p \times {\Z}_p$ we have 
\[ 
\frac{R}{Z} = \langle Z + a, Z + b \;:\; p(Z + a) = p(Z + b) = Z; a, b \in R \rangle, 
\] where $ab \neq ba$. 
If $S/Z$ is additive non-trivial subgroup of $R/Z$ then $|S/Z| = p$. Therefore, any additive proper subgroup of $R$ properly containing $Z$ has $p$ \;disjoint right cosets. Hence, some of the proper additive subgroups of $R$ properly containing $Z$ are
\begin{align*}
S_m &:= C_R(a + mb)\\
&=Z \cup (Z + (a + mb)) \cup (Z + 2(a + mb)) \cup \dots \cup (Z + (p - 1)(a + mb)),\\
       &\text{ where } 1 \leq m \leq (p - 1),\\
S_p &:= C_R(a) = Z \cup (Z + a) \cup (Z + 2a) \cup \dots \cup (Z + (p - 1)a) \text{ and} \\
S_{p + 1} &:= C_R(b) = Z \cup (Z + b) \cup (Z + 2b) \cup \dots \cup (Z + (p - 1)b).      
\end{align*}
Now for any $x \in R \setminus Z$,  we have $Z + x$ is equal to  $Z + k$ for some $k \in \{ma, mb, ma + mb: 1 \leq m \leq (p - 1) \}$. Therefore $C_R(x) = C_R(k)$. Again, let $y \in S_j - Z$ for some $j \in \{1, 2, \dots, (p + 1)\}$, then $C_R(y) \neq S_q$, where $1 \leq q \, (\neq j) \leq (p + 1)$. Thus $C_R(y) = S_j$. Hence $|\Cent(R)| = p + 2$. 

Therefore, by \eqref{comformula} we have 
\[
\Pr(R)  = \frac{|Z|}{p^2|Z|} + \frac{p(p - 1)^2|Z|^2 + 2p(p - 1)|Z|^2}{p^4|Z|^2} = \frac{p^2 + p - 1}{p^3}.
\]    
This completes the proof.
\end{proof}

In \cite{dbn2014}, we have proved the following two results regarding $4$-centralizer and $5$-centralizer finite rings.
\begin{theorem}\label{4,5-cent}
Let $R$ be a finite ring. Then 
\begin{enumerate}
\item[{\rm(a)}] $\frac{R}{Z(R)} \cong {\mathbb{Z}}_2 \times {\mathbb{Z}}_{2}$ if and only if $R$ is $4$-centralizer {\rm(see \cite[Theorem 3.2]{dbn2014})}.
\item[{\rm(b)}]  $\frac{R}{Z(R)} \cong {\mathbb{Z}}_3 \times {\mathbb{Z}}_{3}$ if and only if $R$ is $5$-centralizer {\rm(see \cite[Theorem 4.3]{dbn2014})}.
\end{enumerate}
\end{theorem}
The following result gives an information regarding $6$-centralizer finite rings.
\begin{theorem}\label{z2*z10}
Let $R$ be a finite   ring such that $\frac{R}{Z(R)} \cong {\mathbb{Z}}_2 \times {\mathbb{Z}}_{10}$. Then $R$ is $6$-centralizer. 
\end{theorem}
\begin{proof}
Let $Z := Z(R)$. If $R/Z \cong {\mathbb{Z}}_2 \times {\mathbb{Z}}_{10}$ then there exist two elements $a, b$ of $R$ such that $ab \neq ba$ and 
\[ 
\frac{R}{Z} = \langle Z + a, Z + b \;:\; 2(Z + a) = 10(Z + b) = Z \rangle. 
\]
If $S/Z$ is additive non-trivial subgroup of $R/Z$ then $|S/Z| = 2, 4, 5$ or $10$. Hence, some of the proper additive subgroups of $R$ properly containing $Z$ are
\begin{align*}
S_1 &:= C_R(a) = Z \cup (Z + a),\\
S_2 &:= C_R(b) = Z \cup (Z + b) \cup (Z + 2b) \cup \dots \cup (Z + 9b),\\
S_3 &:= C_R(a + b) \\
 &= Z \cup (Z + (a + b)) \cup  (Z + 2(a + b)) \cup \dots \cup (Z + 9(a + b)),\\
S_4 &:= C_R(a + 2b)\\
 &= Z \cup (Z + (a + 2b)) \cup  (Z + 2(a + 2b)) \cup \dots \cup (Z + 9(a + 2b))\text{ and} \\
S_5 &:= C_R(a + 5b) = Z \cup (Z + (a + 5b)).      
\end{align*}
Now for any $x \in R \setminus Z$,  we have $Z + x$ is equal to  $Z + k$ for some $k \in \{a, mb, a + mb \;:\; 1 \leq m \leq 9 \}$. Therefore $C_R(x) = C_R(k)$. Again, let $y \in S_j - Z$ for some $j \in \{1, 2, \dots, 5\}$, then $C_R(y) \neq S_q$, where $1 \leq q \,(\neq j) \leq 5$. Thus $C_R(y) = S_j$. Thus the proposition follows.
\end{proof}

Following theorem gives some information regarding $8$-centralizer finite rings.

\begin{theorem}\label{z4*z4}
Let $R$ be a finite  ring. If  $\frac{R}{Z(R)} \cong {\mathbb{Z}}_4 \times {\mathbb{Z}}_4, {\mathbb{Z}}_2 \times {\mathbb{Z}}_{12}$ or ${\mathbb{Z}}_3 \times {\mathbb{Z}}_6$ then $R$ is $8$-centralizer.
\end{theorem}
\begin{proof}
Firstly suppose that  $\frac{R}{Z(R)} \cong {\mathbb{Z}}_4 \times {\mathbb{Z}}_4$. Then there exist two elements $a, b \in R$ such that $ab \ne ba$ and
\[
\frac{R}{Z(R)} = \langle a + Z, b + Z : 4(a + Z) = 4(b + Z) = Z \rangle
\] 
where $Z := Z(R)$. 
If $S/Z$ is additive non-trivial subgroup of $R/Z$ then $|S/Z| = 2, 4$ or $8$. Therefore, some of the proper additive subgroups of $R$ properly containing $Z$ are
\begin{align*}
S_m &:= C_R(a + mb) \\
&= Z \cup (Z + (a + mb)) \cup (Z + 2(a + mb)) \cup (Z + 3(a + mb)),\\
      & \qquad \text{where } 1 \leq m \leq 3,\\
S_4 &:= C_R(2a + b)\\
&= Z \cup (Z + (2a + b)) \cup (Z + 2(2a + b))) \cup (Z + 3(2a + b)),\\
S_5 &:= C_R(2a +2b) = Z \cup  (Z + (2a + 2b)), \\            
S_6 &:= C_R(a) = Z \cup (Z + a) \cup (Z + 2a) \cup (Z + 3a) \text{ and} \\
S_7 &:= C_R(b) = Z \cup (Z + b) \cup (Z + 2b) \cup (Z + 3b).      
\end{align*}
Now for any $x \in R \setminus Z$,  we have $Z + x$ is equal to  $Z + k$ for some $k \in \{ma, mb, ma + mb: 1 \leq m \leq 3 \}$. Therefore $C_R(x) = C_R(k)$. Again, let $y \in S_j - Z$ for some $j \in \{1, 2, \dots, 7\}$, then $C_R(y) \neq S_q$, where $1 \leq q \, (\neq j) \leq 7$. Thus $C_R(y) = S_j$. This shows that $R$ is  $8$-centralizer.

If $R/Z \cong {\mathbb{Z}}_2 \times {\mathbb{Z}}_{12}$ then there exist two elements $a, b$ of $R$ such that $ab \neq ba$ and 
\[ 
\frac{R}{Z} = \langle Z + a, Z + b \;:\; 2(Z + a) = 12(Z + b) = Z \rangle. 
\]
If $S/Z$ is additive non-trivial subgroup of $R/Z$ then $|S/Z| = 2, 3, 4, 6, 8$ or $12$. Hence, some of the proper additive subgroups of $R$ properly containing $Z$ are
\allowdisplaybreaks{
\begin{align*}
S_1 &:= C_R(a) = Z \cup (Z + a),\\
S_2 &:= C_R(b) = Z \cup (Z + b) \cup (Z + 2b) \cup \dots \cup (Z + 9b),\\
S_3 &:= C_R(a + b)\\
 &= Z \cup (Z + (a + b)) \cup  (Z + 2(a + b)) \cup \dots \cup (Z + 11(a + b)),\\
S_4 &:= C_R(a + 2b) \\
 &= Z \cup (Z + (a + 2b)) \cup  (Z + 2(a + 2b)) \cup \dots \cup (Z + 5(a + 2b)),\\
S_5 &:= C_R(a + 4b)\\
 &= Z \cup (Z + (a + 4b)) \cup  (Z + 2(a + 4b)) \cup \dots \cup (Z + 5(a + 4b)),\\
S_6 &:= C_R(a + 3b)\\
 &= Z \cup (Z + (a + 3b)) \cup  (Z + 2(a + 3b)) \cup (Z + 3(a + 3b)) \text{ and} \\
S_7 &:= C_R(a + 6b) = Z \cup (Z + (a + 6b)).      
\end{align*}
}
Now for any $x \in R \setminus Z$,  we have $Z + x$ is equal to  $Z + k$ for some $k \in \{a, mb, a + mb \;:\; 1 \leq m \leq 11 \}$. Therefore $C_R(x) = C_R(k)$. Again, let $y \in S_j - Z$ for some $j \in \{1, 2, \dots, 7\}$, then $C_R(y) \neq S_q$, where $1 \leq q \,(\neq j) \leq 7$. Thus $C_R(y) = S_j$. This shows that $R$ is  $8$-centralizer.

If $\frac{R}{Z(R)} \cong {\mathbb{Z}}_3 \times {\mathbb{Z}}_6$ then there exist two elements $a, b \in R$ such that $ab \ne ba$ and
\[
\frac{R}{Z(R)} = \langle a + Z, b + Z : 3(a + Z) = 6(b + Z) = Z \rangle
\] 
where $Z := Z(R)$. 
If $S/Z$ is additive non-trivial subgroup of $R/Z$ then $|S/Z| = 3, 6$ or $9$. Therefore, some of the proper additive subgroups of $R$ properly containing $Z$ are
\allowdisplaybreaks{
\begin{align*}
S_1 &:= C_R(a) = Z \cup (Z + a) \cup (Z + 2a),\\
S_2 &:= C_R(b) = Z \cup (Z + b) \cup (Z + 2b) \cup \dots \cup (Z + 5b),\\
S_3 &:= C_R(a + b)\\
&= Z \cup (Z + (a + b)) \cup (Z + 2(a + b)) \cup \dots \cup 
(Z + 5(a + b)),\\
S_4 &:= C_R(a + 2b) = Z \cup (Z + (a + 2b)) \cup (Z + 2(a + 2b)),\\
S_5 &:= C_R(a + 3b)\\
&= Z \cup (Z + (a + 3b)) \cup (Z + 2(a + 3b)) \cup \dots \cup (Z + 5(a + 3b)),\\            
S_6 &:= C_R(a + 4b) = Z \cup (Z + (a + 4b)) \cup (Z + 2(a + 4b)) \text{ and } \\
S_7 &:= C_R(a + 5b)\\
&= Z \cup (Z + (a + 5b)) \cup (Z + 2(a + 5b)) \cup \dots \cup (Z + 5(a + 5b)).      
\end{align*}
}
Now for any $x \in R \setminus Z$,  we have $Z + x$ is equal to  $Z + k$ for some $k \in \{ma, nb, ma + nb : m \in \{1, 2\}, 1 \leq n \leq 5 \}$. Therefore $C_R(x) = C_R(k)$. Again, let $y \in S_j - Z$ for some $j \in \{1, 2, \dots, 7\}$, then $C_R(y) \neq S_q$, where $1 \leq q \,(\neq j) \leq 7$. Thus $C_R(y) = S_j$. This shows that $R$ is  $8$-centralizer.
\end{proof}

\noindent We conclude this section with the following result which proves the non-existence of rings with central factor isomorphic to ${\mathbb{Z}}_2 \times {\mathbb{Z}}_4$.

\begin{theorem}\label{z2*z4}
There is no finite non-commutative ring $R$ such that $\frac{R}{Z(R)}$ is isomorphic to ${\mathbb{Z}}_2 \times {\mathbb{Z}}_4$.
\end{theorem}
\begin{proof}
Let $R$ be a finite non-commutative ring such that $\frac{R}{Z(R)} \cong {\mathbb{Z}}_2 \times {\mathbb{Z}}_4$. Then there exist two elements $a, b \in R$ such that $ab \ne ba$ and
\[
\frac{R}{Z(R)} = \langle a + Z, b + Z : 2(a + Z) = 4(b + Z) = Z\rangle
\] 
where $Z := Z(R)$.
If $S/Z$ is additive non-trivial subgroup of $R/Z$ then $|S/Z| = 2$ or $4$. So, some of the proper additive subgroups of $R$ properly containing $Z$ are
\begin{align*}
S_1 &:= C_R(a + b)\\
&= Z \cup (Z + (a + b)) \cup (Z + 2(a + b)) \cup (Z + 3(a + b)),\\
S_2 &:= C_R(a + 2b) = Z \cup (Z + (a + 2b)), \\     
S_3 &:= C_R(a) = Z \cup (Z + a) \text{ and} \\
S_4 &:= C_R(b)= Z \cup (Z + b) \cup (Z + 2b) \cup (Z + 3b).      
\end{align*}
Now for any $x \in R \setminus Z$,  we have $Z + x$ is equal to  $Z + k$ for some $k \in \{a, mb, a + mb: 1 \leq m \leq 3 \}$. Therefore $C_R(x) = C_R(k)$. Again, let $y \in S_j - Z$ for some $j \in \{1, 2, 3, 4\}$, then $C_R(y) \neq S_q$, where $1 \leq q \, (\neq j) \leq 4$. Thus $C_R(y) = S_j$. Thus $R$ is a $5$-centralizer ring. Hence, by Theorem $4.3$ in \cite{dbn2014}, we have $\frac{R}{Z(R)} \cong {\mathbb{Z}}_3 \times {\mathbb{Z}}_3$, a contradiction. This completes the proof.
\end{proof}

\section{Commuting probabilities}
In this section, we compute the commuting probabilities of $n$-centralizer finite rings for $n \leq 7$. It is clear that $R$ is $1$-centralizer if and only if it is commutative. Therefore, $R$ is $1$-centralizer if and only if $\Pr(R) = 1$.
In {\rm \cite[Theorem 2.6]{dbn2014}}, we have proved that there is no $n$-centralizer ring for $n = 2, 3$. The following two results give  commuting probabilities of $4$-centralizer and $5$-centralizer finite rings. 
\begin{theorem}
Let $R$ be a  finite   ring. Then $R$ is $4$-centralizer if and only if $\Pr(R) = \frac{5}{8}$.
\end{theorem}
\begin{proof}
The proof follows from Theorem \ref{4,5-cent}(a) and Theorem \ref{3chdc}.
\end{proof}

\begin{theorem}
If $R$ is a finite $5$-centralizer ring then $\Pr(R) = \frac{11}{27}$. 
\end{theorem}
\begin{proof}
The proof follows from Theorem \ref{4,5-cent}(b) and Theorem \ref{machale1-2}.
\end{proof}

The following characterization of finite $6$-centralizer  rings is    useful in computing their commuting probabilities.

\begin{theorem}\label{classification6c}
If $R$ is a $6$-centralizer finite ring then $\frac{R}{Z(R)}$ is isomorphic to
${\mathbb{Z}}_2 \times {\mathbb{Z}}_2 \times {\mathbb{Z}}_2, {\mathbb{Z}}_2 \times {\mathbb{Z}}_6, {\mathbb{Z}}_2 \times {\mathbb{Z}}_8, {\mathbb{Z}}_2 \times {\mathbb{Z}}_2 \times {\mathbb{Z}}_4 \text{ or } {\mathbb{Z}}_2 \times {\mathbb{Z}}_2 \times {\mathbb{Z}}_2 \times {\mathbb{Z}}_2.
$
\end{theorem}
\begin{proof}
If $R$ is a $6$-centralizer finite ring then by Theorem $3.1$ in \cite{dbn2015}, we have $|R : Z(R)| = 8, 12$ or $16$. 
By  fundamental theorem of finite abelian groups, $\frac{R}{Z(R)}$ is isomorphic to ${\mathbb{Z}}_8, {\mathbb{Z}}_2 \times {\mathbb{Z}}_4$, ${\mathbb{Z}}_2 \times {\mathbb{Z}}_2 \times {\mathbb{Z}}_2, {\mathbb{Z}}_{12}$, ${\mathbb{Z}}_2 \times {\mathbb{Z}}_6$, ${\mathbb{Z}}_{16}, {\mathbb{Z}}_2 \times {\mathbb{Z}}_8, {\mathbb{Z}}_4 \times {\mathbb{Z}}_4, {\mathbb{Z}}_2 \times {\mathbb{Z}}_2 \times {\mathbb{Z}}_4$ or ${\mathbb{Z}}_2 \times {\mathbb{Z}}_2 \times {\mathbb{Z}}_2 \times {\mathbb{Z}}_2$.

 Since $R$ is non-commutative,  $\frac{R}{Z(R)}$ is not isomorphic to ${\mathbb{Z}}_8, {\mathbb{Z}}_{12}$ or ${\mathbb{Z}}_{16}$. Also, by Theorem \ref{z4*z4} and Theorem \ref{z2*z4}, we have $\frac{R}{Z(R)}$ is not isomorphic to ${\mathbb{Z}}_4 \times {\mathbb{Z}}_4$ and  ${\mathbb{Z}}_2 \times {\mathbb{Z}}_4$. Hence the result follows. 
\end{proof}

\begin{theorem}\label{classificationPr6c}
If $R$ is a $6$-centralizer finite ring such that $\frac{R}{Z(R)}$ is not isomorphic to ${\mathbb{Z}}_2 \times {\mathbb{Z}}_2 \times {\mathbb{Z}}_4$ and ${\mathbb{Z}}_2 \times {\mathbb{Z}}_2 \times {\mathbb{Z}}_2 \times {\mathbb{Z}}_2$ then $\Pr(R) \in \{\frac{7}{16}, \frac{35}{72}, \frac{29}{64}\}$.
\end{theorem}
\begin{proof}
Since $R$ is $6$-centralizer finite ring,  by Theorem \ref{classification6c}, we have 
\[
\frac{R}{Z} \cong {\mathbb{Z}}_2 \times {\mathbb{Z}}_2 \times {\mathbb{Z}}_2, {\mathbb{Z}}_2 \times {\mathbb{Z}}_6 \text{ or } {\mathbb{Z}}_2 \times {\mathbb{Z}}_8
\]
where $Z := Z(R)$.

 If $\frac{R}{Z} \cong {\mathbb{Z}}_2 \times {\mathbb{Z}}_2 \times {\mathbb{Z}}_2$ then there exist three non-central elements $a, b, c$ of $R$ not commuting with each other simultaneously and 
\[ 
\frac{R}{Z} = \langle Z + a, Z + b, Z + c \;:\; 2(Z + a) = 2(Z + b) = 2(Z + c) = Z \rangle. 
\]
If $S/Z$ is additive non-trivial subgroup of $R/Z$ then $|S/Z| = 2$ or $4$. Since $|\Cent(R)| = 6$, we have
\[
C_R(k) = C_R(l) = C_R(k + l) = Z \cup (Z + k) \cup (Z + l) \cup (Z + (k + l))
\]
 for some $k, l \in \{a, b, c, a + b, a + c, b + c, a + b + c\}$ and $k \neq l$. Without loss of generality, we can assume that $C_R(b) = C_R(c)$. Hence, some of the proper additive subgroups of $R$ properly containing $Z$ are
\begin{align*}
S_1 &:= C_R(a) = Z \cup (Z + a),\\
S_2 &:= C_R(b) = C_R(c) = C_R(b + c) = Z \cup (Z + b) \cup (Z + c) \cup (Z + (b + c)), \\
S_3 &:= C_R(a + b)= Z \cup (Z + (a + b)),\\
S_4 &:= C_R(a +c) = Z \cup (Z + (a + c))\text{ and} \\
S_5 &:= C_R(a + b + c) = Z \cup (Z + (a + b + c)).      
\end{align*}
Now for any $x \in R \setminus Z$,  we have $Z + x$ is equal to  $Z + k$ for some $k \in \{a, b, c, a + b, a + c, b + c, a + b + c \}$. Therefore $C_R(x) = C_R(k)$. Again, let $y \in S_j - Z$ for some $j \in \{1, 2, \dots, 5\}$, then $C_R(y) \neq S_q$, where $1 \leq q \,(\neq j) \leq 5$. Thus $C_R(y) = S_j$. Hence, by \eqref{comformula} we have 
\[
\Pr(R)  = \frac{|Z|}{8|Z|} + \frac{(4 \times 2) |Z|^2 + (3 \times 4) |Z|^2}{8^2|Z|^2} = \frac{7}{16}.
\] 
If $\frac{R}{Z} \cong {\mathbb{Z}}_2 \times {\mathbb{Z}}_6$ then there exist two elements $a, b$ of $R$ such that $ab \neq ba$ and 
\[ 
\frac{R}{Z} = \langle Z + a, Z + b \;:\; 2(Z + a) = 6(Z + b) = Z \rangle. 
\]
If $S/Z$ is additive non-trivial subgroup of $R/Z$ then $|S/Z| = 2, 3, 4$ or $6$. Hence, some of the proper additive subgroups of $R$ properly containing $Z$ are
\begin{align*}
S_1 &:= C_R(a) = Z \cup (Z + a),\\
S_2 &:= C_R(b) = Z \cup (Z + b) \cup (Z + 2b) \cup \dots \cup (Z + 5b), \\
S_3 &:= C_R(a + b) = Z \cup (Z + (a + b)) \cup  (Z + 2(a + b)) \cup \dots \cup (Z + 5(a + b)),\\
S_4 &:= C_R(a + 2b)\\
&= Z \cup (Z + (a + 2b)) \cup  (Z + 2(a + 2b)) \cup \dots \cup (Z + 5(a + 2b)) \text{ and} \\
S_5 &:= C_R(a + 3b) = Z \cup (Z + (a + 3b)).      
\end{align*}
Now for any $x \in R \setminus Z$,  we have $Z + x$ is equal to  $Z + k$ for some $k \in \{a, mb, a + mb \;:\; 1 \leq m \leq 5 \}$. Therefore $C_R(x) = C_R(k)$. Again, let $y \in S_j - Z$ for some $j \in \{1, 2,   \dots, 5\}$, then $C_R(y) \neq S_q$, where $1 \leq q            \,(\neq j) \leq 5$. Thus $C_R(y) = S_j$. Hence, by \eqref{comformula} we have 
\[
\Pr(R)  = \frac{|Z|}{12|Z|} + \frac{(2 \times 2) |Z|^2 + (9 \times 6) |Z|^2}{{12}^2|Z|^2} = \frac{35}{72}.
\]   
If $\frac{R}{Z} \cong {\mathbb{Z}}_2 \times {\mathbb{Z}}_8$ then there exist two elements $a, b$ of $R$ such that $ab \neq ba$ and 
\[ 
\frac{R}{Z} = \langle Z + a, Z + b \;:\; 2(Z + a) = 8(Z + b) = Z; a, b \in R \rangle. 
\]
If $S/Z$ is additive non-trivial subgroup of $R/Z$ then $|S/Z| = 2, 4$ or $8$. Hence, some of the proper additive subgroups of $R$ properly containing $Z$ are

\begin{align*}
S_1 &:= C_R(a) = Z \cup (Z + a),\\
S_2 &:= C_R(b) = Z \cup (Z + b) \cup (Z + 2b) \cup \dots \cup (Z + 7b), \\
S_3 &:= C_R(a + b) = Z \cup (Z + (a + b)) \cup  (Z + 2(a + b)) \cup \dots \cup (Z + 7(a + b)),\\
S_4 &:= C_R(a + 2b)\\
 &= Z \cup (Z + (a + 2b)) \cup  (Z + 2(a + 2b)) \cup \dots \cup (Z + 4(a + 2b))\text{ and} \\
S_5 &:= C_R(a + 4b) = Z \cup (Z + (a + 4b)).      
\end{align*}
Now for any $x \in R \setminus Z$,  we have $Z + x$ is equal to  $Z + k$ for some $k \in \{a, mb, a + mb \;:\; 1 \leq m \leq 7 \}$. Therefore $C_R(x) = C_R(k)$. Again, let $y \in S_j - Z$ for some $j \in \{1, 2, \dots, 5\}$, then $C_R(y) \neq S_q$, where $1 \leq q \,(\neq j) \leq 5$. Thus $C_R(y) = S_j$. Hence, by \eqref{comformula} we have 
\[
\Pr(R)  = \frac{|Z|}{16|Z|} + \frac{(2 \times 2) |Z|^2 + (11 \times 8) |Z|^2 + (2 \times 4) |Z|^2}{{16}^2|Z|^2} = \frac{29}{64}.
\] 
This completes the proof.
\end{proof}

Now we compute  commuting probability of some finite $7$-centralizer rings. The following result is   useful in this regard.

\begin{proposition}\label{7cent_bc}
Let $R$ be a finite $7$-centralizer ring. Then 
\[
\frac{R}{Z(R)} \cong {\mathbb{Z}}_2 \times {\mathbb{Z}}_2 \times {\mathbb{Z}}_6 \text{ or } {\mathbb{Z}}_5 \times {\mathbb{Z}}_5.
\]
\end{proposition}
\begin{proof}
If $R$ is a $7$-centralizer finite ring then by Theorem $4.3$ in \cite{dbn2015}, we have $|R : Z(R)| = 12, 18, 20$, $24$ or $25$.
Therefore by fundamental theorem of finite abelian groups,  $\frac{R}{Z(R)}$ is isomorphic to
${\mathbb{Z}}_{12},  {\mathbb{Z}}_2 \times {\mathbb{Z}}_6, {\mathbb{Z}}_{18}$, ${\mathbb{Z}}_3 \times {\mathbb{Z}}_6$, ${\mathbb{Z}}_{20},  {\mathbb{Z}}_2 \times {\mathbb{Z}}_{10}$, ${\mathbb{Z}}_{24}, {\mathbb{Z}}_2 \times {\mathbb{Z}}_{12}, {\mathbb{Z}}_2 \times {\mathbb{Z}}_2 \times {\mathbb{Z}}_6$, ${\mathbb{Z}}_{25} \text{ or } {\mathbb{Z}}_5 \times {\mathbb{Z}}_5$.

 Since $R$ is non-commutative,   $\frac{R}{Z(R)}$ is not isomorphic to ${\mathbb{Z}}_{12}, {\mathbb{Z}}_{18}, {\mathbb{Z}}_{20}, {\mathbb{Z}}_{24}$ and ${\mathbb{Z}}_{25}$.
It is shown in the proof of Theorem \ref{classificationPr6c} that that $R$ is $6$-centralizer if $\frac{R}{Z(R)} \cong {\mathbb{Z}}_2 \times {\mathbb{Z}}_6$. Therefore, $\frac{R}{Z(R)}$ is not isomorphic to ${\mathbb{Z}}_2 \times {\mathbb{Z}}_6$.
Also, by Theorem \ref{z2*z10} and Theorem \ref{z4*z4}, 
we have $\frac{R}{Z(R)}$ is not isomorphic to ${\mathbb{Z}}_2 \times {\mathbb{Z}}_{10}, {\mathbb{Z}}_3 \times {\mathbb{Z}}_6$ and ${\mathbb{Z}}_2 \times {\mathbb{Z}}_{12}$.
Hence, the theorem follows.
\end{proof}

We conclude this paper by the following result.
\begin{theorem}\label{pr7c}
Let $R$ be a finite $7$-centralizer ring such that $R/Z(R)$ is not isomorphic to ${\mathbb{Z}}_2 \times {\mathbb{Z}}_2 \times {\mathbb{Z}}_6$. Then $\Pr(R) = \frac{29}{125}$.
\end{theorem}
\begin{proof}
The result follows from Theorem \ref{7cent_bc} and Theorem \ref{3chdc}.
\end{proof}



\end{document}